\documentclass[english,reqno]{amsart}

\usepackage{amsmath}
\usepackage{amssymb} 
\usepackage{enumitem} 
\usepackage{mathtools} 
\usepackage[table]{xcolor} 
\usepackage[all]{xy} 
\usepackage{tikz} 
\usepackage{tikz-cd}
\usepackage{indentfirst} 
\usepackage{babel} 
\usepackage{setspace} 

\usepackage[colorlinks,linkcolor=red,anchorcolor=green,citecolor=blue]{hyperref} 
\hypersetup{linktocpage = true} 

\usepackage{rotating} 

\usepackage{ytableau} 
\usepackage{longtable} 
\newcolumntype{M}[1]{>{\centering\arraybackslash}m{#1}} 

\usepackage{mathpazo}
\usepackage{mathrsfs}
\DeclareFontFamily{OMS}{rsfs}{\skewchar\font'60}
\DeclareFontShape{OMS}{rsfs}{m}{n}{<-5>rsfs5 <5-7>rsfs7 <7->rsfs10 }{}
\DeclareSymbolFont{rsfs}{OMS}{rsfs}{m}{n}
\DeclareSymbolFontAlphabet{\scr}{rsfs}
\DeclareSymbolFontAlphabet{\scr}{rsfs}

\usepackage[T1]{fontenc}

\newcommand\bbC{{\mathbb C}}

\newcommand\bbP{{\mathbb P}}

\newcommand\bbZ{{\mathbb Z}}

\newcommand\sF{{\mathscr F}}

\newcommand\sO{{\mathscr O}}

\DeclareMathOperator*{\Sym}{Sym}

\theoremstyle{plain}
\newtheorem{thm}{Theorem}[section]
\newtheorem{lemma}[thm]{Lemma}
\newtheorem{prop}[thm]{Proposition}
\newtheorem{cor}[thm]{Corollary}

\newtheorem{conj}[thm]{Conjecture}

\newtheorem{problem}[thm]{Problem}

\theoremstyle{definition}
\newtheorem{example}[thm]{Example}
\newtheorem{remark}[thm]{Remark}

\setlist[itemize]{leftmargin=*}
\setlist[enumerate]{leftmargin=*}

\numberwithin{equation}{section} 

\makeatletter

\ifnum\@ptsize=0  \fi \ifnum\@ptsize=2  \fi

\title{Strictly nef vector bundles and characterizations of $\bbP^n$}
\date{\today}


\author{Jie Liu}

\address{Jie Liu, Morningside Center of Mathematics, Academy of Mathematics and Systems Science, Chinese Academy of Sciences, Beijing, 100190, China}
\email{jliu@amss.ac.cn}

\author{Wenhao Ou}

\address{Wenhao Ou, Institute of Mathematics, Academy of Mathematics and Systems Science, Chinese Academy of Sciences, Beijing, 100190, China}
\email{wenhaoou@amss.ac.cn}

\author{Xiaokui Yang}

\address{Xiaokui Yang, Department of Mathematics and Yau Mathematical Sciences Center, Tsinghua University, Beijing, 100084, China}
\email{xkyang@mail.tsinghua.edu.cn}

\begin{document}

\begin{abstract}  In this note, we give a brief exposition on the differences and similarities between strictly nef and ample vector bundles, with particular focus on the circle of  problems surrounding the geometry of projective manifolds  with strictly nef bundles.
\end{abstract}

\maketitle

\tableofcontents

\vspace{-0.2cm}

\section{Introduction}

Let $X$ be a complex projective manifold. A line bundle $L$ over $X$ is said to be \emph{strictly nef} if   $$L\cdot
C>0$$ for each irreducible curve $C\subset X$.  This notion is also called "numerically positive" in literatures (e.g. \cite{Hartshorne1970}).  The Nakai-Moishezon-Kleiman criterion asserts that  $L$ is ample if and only if
$$ L^{\dim Y}\cdot Y>0 $$ for every positive-dimensional irreducible subvariety $Y$ in $X$.  Hence, ample line bundles are strictly nef. In 1960s, Mumford constructed a number of  strictly nef but non-ample line bundles over  ruled surfaces (e.g. \cite{Hartshorne1970}), and they are  tautological line bundles of stable vector bundles of degree zero over smooth curves of genus $g\geq 2$. By using the terminology of Harshorne (\cite{Hartshorne1966}), a vector bundle $E\rightarrow X$ is called strictly nef (resp. ample) if its tautological line bundle $\sO_{\bbP(E)}(1)$ is strictly nef (resp. ample). One can see immediately that the strictly nef vector bundles constructed by Mumford are actually \emph{Hermitian-flat}. Therefore, some functorial properties for ample bundles (\cite{Hartshorne1966}) are not valid for strictly nef bundles. In this note, we give a brief exposition on the differences and similarities between strict nefness and ampleness, and survey some recent progress on understanding the geometry of projective manifolds endowed with some strictly nef bundles.

Starting in the mid 1960’s, several mathematicians--notably Grauert, Griffiths  and Hartshorne (\cite{Grauert1962, Griffiths1965, Griffiths1969,Hartshorne1966})--undertook the task
of generalizing to vector bundles the theory of positivity for line bundles. One
of the goals was to extend to the higher rank setting as many as possible of the
beautiful cohomological and topological properties enjoyed by ample divisors. In the past  half-century, a number of  fundamental results have been established. For this rich topic,  we refer to the books \cite{Lazarsfeld2004,Lazarsfeld2004a} of Lazarsfeld and the references therein.

\subsection{Abstract strictly nef vector bundles.} Let's recall a criterion for
 strictly nef vector bundles (see \cite[Proposition 2.1]{LiOuYang2019})  which is analogous to the Barton-Kleiman criterion for nef vector bundles (e.g. \cite[Proposition ~6.1.18]{Lazarsfeld2004a}). This criterion will be used frequently in the sequel.

 \begin{prop}\label{cri0} Let $E$ be a vector bundle over a projective manifold $X$. Then the following conditions are equivalent.
    \begin{enumerate}
        \item[(1)] $E$ is  strictly nef.

        \item[(2)] For any    smooth projective curve $C$ with a  finite morphism $\nu:C\rightarrow X$,  and for any line bundle quotient $\nu^*(E)\to L$, one has $$ \deg L>0.  $$

    \end{enumerate}

 \end{prop}

 \noindent Recall that for an ample line bundle $L$, one has the Kodaira vanishing theorem
 $$H^i(X,L^*)=0\ \  \text{for\ }\  i<\dim X.$$
 For an ample vector bundle $E$ with rank $r\geq 2$, one can only deduce $$H^0(X,E^*)=0$$ and the higher cohomology groups $H^i(X,E^*)$ ($i\geq 1$) may not  vanish.
 By using Proposition \ref{cri0}, we obtain a similar vanishing theorem  for strictly nef vector bundles.

 \begin{thm} Let $E$ be a strictly nef vector bundle over a projective manifold $X$. Then $$H^0(X,E^*)=0.$$
 \end{thm}

\noindent It worths to point out that for a strictly nef vector bundle $E$ with rank $r\geq 2$, the cohomology group $H^0(X,\mathrm{Sym}^{\otimes k} E^*)$ may not vanish for $k\geq 2$, which is significantly different from properties of ample vector bundles. Indeed, the strict nefness is not closed under tensor product, symmetric product or exterior product of vector bundles. These will be discussed in Section \ref{basic}.\\

 It is well-known that vector bundles over $\bbP^1$ split into direct sums of line bundles. By using Proposition \ref{cri0} again, one deduces that strictly nef vector bundles over $\bbP^1$ are  ample. In \cite[Theorem~3.1]{LiOuYang2019}, the following result is obtained.
 \begin{cor} If $E$ is a strictly nef vector bundle over an elliptic curve $C$, then $E$ is ample.
 \end{cor}
 \noindent  As we mentioned before, over smooth curves of genus $g\geq 2$,  there are strictly nef vector bundles  which are Hermitian-flat.  There also exist strictly nef but non-ample  bundles on some rational surfaces (\cite{LanteriRondena1994, Chaudhuri2020}).  It is still a challenge to investigate  strictly nef vector bundles over higher dimensional projective manifolds.
 We propose the following conjecture, which is also the first step to understand such bundles.

 \begin{conj}\label{LOYconjecture1} Let $E$ be a strictly nef vector bundle over a projective manifold $X$.   If $-K_X$ is nef, then $\det E$ is ample.
 \end{conj}

 \noindent Although this conjecture is shown to be a consequence of the "\emph{generalized abundance conjecture}" (e.g. \cite{LazicPeternell2020,LazicPeternell2020a}), we still expect some other straightforward solutions. Indeed, we get  a partial answer to it.

 \begin{thm} Let $E$ be a strictly nef vector bundle over a projective manifold $X$.   If $-K_X$ is nef and big, then $\det E$ is ample.
 \end{thm}
\noindent  We  refer to \cite{Hering2010,Biswas2019} for more details on positivity of equivariant vector bundles.

\subsection{The geometry of projective manifolds endow with strictly nef bundles.} Since the seminal works of Mori   and Siu-Yau
(\cite{Mori1979}, \cite{SiuYau1980}) on characterizations of projective spaces, it becomes apparent that the
positivity
of the tangent bundle of a complex projective manifold
carries important geometric information. In the past
decades,  many remarkable generalizations  have been
established, and for instances, Mok's uniformization theorem on compact K\"ahler manifold
with semipositive holomorphic bisectional curvature (\cite{Mok1988}) and  fundamental works of Campana,
Demailly, Peternell and Schneider (\cite{CampanaPeternell1991},
\cite{DemaillyPeternellSchneider1994}, \cite{Peternell1996})  on the
structure of projective manifolds with nef tangent bundles. For this comprehensive topic, we refer to \cite{Yau1974,Mori1979, SiuYau1980,Mok1988,CampanaPeternell1991,DemaillyPeternellSchneider1994,Peternell1996,DemaillyPeternellSchneider2001,MunozOcchettaSolaCondeWatanabeEtAl2015,Cao2016,LiOuYang2019,CaoHoering2019,CaoHoering2019,CampanaCaoPaun2019} and the references therein. \\

As we pointed out before,  strict nefness is a notion of positivity weaker than ampleness.   Even though there are significant differences between them, we still expect that  the strict nefness could play similar roles as ampleness in many  situations. The following result is obtained in \cite[Theorem~1.4]{LiOuYang2019} which extends Mori's Theorem.
\begin{thm} Let $X$ be a projective manifold. If $T_X$ is strictly nef, then $X$ is isomorphic to a projective space.
\end{thm}

\noindent Therefore, $T_X$ is ample if and only if it is strictly nef.  However, this is not valid for cotangent bundles. Indeed, Shepherd-Barron proved in  \cite{ShepherdBarron1995} that there exists a projective surface whose cotangent bundle is strictly nef but not ample (see e.g. Example \ref{cotangentbundle}).\\

Let's consider manifolds with strictly nef canonical or anti-canonical bundles. Campana and Peternell proposed in
\cite[Problem~11.4]{CampanaPeternell1991} the following conjecture, which is still a major problem along this line.

\begin{conj} \label{CampanaPeternellConjecture1} Let $X$ be a  projective
    manifold. If $K_X^{-1}$ is strictly nef, then $X$
    is  Fano.
\end{conj}

\noindent This conjecture has been verified for projective manifolds of dimension $2$ in \cite{Maeda1993} and dimension $3$ in \cite{Serrano1995} (see also \cite{Uehara2000} and the references therein).
Recently, some progress has been achieved in \cite[Theorem~1.2]{LiOuYang2019}.

\begin{thm}\label{LOYrationallyconnected1}  If $K^{-1}_X$ is strictly nef, then $X$ is rationally connected.
\end{thm}

\noindent  Indeed, we show in \cite{LiuOuYang2} that if $(X,\Delta)$ is a projective simple normal crossing pair and $-(K_X+\Delta)$ is strictly nef, then $X$ is rationally connected. 

The following dual version of Conjecture \ref{CampanaPeternellConjecture1} is actually a consequence of the abundance conjecture.

\begin{conj} If $K_X$ is strictly nef, then $K_X$ is ample.
    \end{conj}

\noindent  As analogous to the Fujita conjecture, Serrano proposed in \cite{Serrano1995} the following conjecture, which is a generalization of Conjecture \ref{CampanaPeternellConjecture1}.

\begin{conj}\label{Serranoconjecture1} Let $X$ be a projective manifold. If $L$ is a strictly nef line bundle, then $K_X\otimes L^{\otimes m}$ is ample for  $m\geq \dim X+2$.
\end{conj}

\noindent This conjecture has been solved for projective surfaces in \cite{Serrano1995}. For the progress on projective threefolds and higher dimensional manifolds, we refer to \cite[Theorem~0.4]{CampanaChenPeternell2008} and the references therein.\\

It is also known that the existence of "positive" subsheaves of the
tangent bundle can also characterize  the
ambient manifold. For instance,  Andreatta and Wi\'sniewski obtained in \cite[Theorem]{AndreattaWisniewski2001}
the following characterization of projective spaces:
\begin{thm}
    \label{thm:Andreatta-Wisniewski1}
    Let $X$ be a projective manifold. If the tangent bundle $T_X$ contains a locally free ample subsheaf $\sF$, then $X$ is isomorphic to a projective space.
\end{thm}

\noindent When $\sF$ is a line bundle, this result is proved by
Wahl in \cite{Wahl1983}, and in \cite{CampanaPeternell1998}, Campana and
Peternell established the cases $r\geqslant n-2$. It is also shown that the assumption on the  local freeness can be
dropped (\cite{AproduKebekusPeternell2008,Liu2019}). On the other hand, according to Mumford's construction (see Example \ref{Mumfordexample}), Theorem \ref{thm:Andreatta-Wisniewski1}
does not hold  if the subsheaf $\sF$ is  assumed to be strictly nef. Indeed, we  obtained in \cite[Theorem~1.3]{LiuOuYang2020} the following  result, which is an extension of
Theorem \ref{thm:Andreatta-Wisniewski1}.

\begin{thm} \label{thm:main-theorem1}
	
	Let $X$ be a   projective manifold. Assume that the tangent bundle $T_X$ contains a locally free strictly nef subsheaf $\sF$ of rank $r$.
	Then $X$ admits a $\bbP^d$-bundle structure $\varphi\colon X\rightarrow T$ for some integer $d\geq r$.
	Furthermore, if $T$ is not a single point, then it is a hyperbolic projective manifold of general type.
\end{thm}

\noindent Recall that  a single point is also considered to be hyperbolic in the sense that every holomorphic map from $\bbC$ to it is  constant. We expect a stronger geometric positivity on the cotangent bundle of the base $Y$ in Theorem \ref{thm:main-theorem1} when $\dim Y>0$.  
As an application of Theorem \ref{thm:main-theorem1}, we obtain in \cite[Theorem~1.4]{LiuOuYang2020} a characterization  of projective spaces.

\begin{thm}
	\label{thm:simply-connected-Pn}  Let $X$ be an $n$-dimensional  complex projective manifold such that $T_X$ contains a locally free strictly nef subsheaf $\sF$.
	If $\pi_1(X)$ is virtually solvable, then $X$ is isomorphic to  $\bbP^n$, and $\sF$ is isomorphic to either $T_{\bbP^n}$ or $\sO_{\bbP^n}(1)^{\oplus r}$.
\end{thm}

\noindent There are many other characterizations of projective spaces, and we refer to \cite{KobayashiOchiai1973,Peternell1990,CampanaPeternell1998,Druel2004,Hwang2006,Araujo2006,AraujoDruelKovacs2008, Yang2017,Tosatti2017,FengLiuWan2017,Li2018} and the references therein.\\

\noindent{\bf Acknowledgements.} 
 The first-named author is supported by China Postdoctoral Science Foundation (2019M650873).

\vskip 2\baselineskip

\section{Basic properties and examples}\label{basic}

In this section, we investigate basic properties of strictly nef bundles and discuss some examples. As we mentioned before, Mumford
constructed a strictly nef vector bundle which is  not ample  (see \cite[Chapter I, Example 10.6]{Hartshorne1970}). We shall describe this example in details. Let $E$ be a rank $2$ vector bundle over a smooth curve $C$ of genus $g\geq 2$,   $X=\bbP(E)$ be the projectivized bundle  and $\pi:\bbP(E)\rightarrow C$ be the projection. Let
$\sO_E(1)$ be the tautological line bundle of $\bbP(E)$ and $D$ be
the corresponding divisor over $X$.

\begin{lemma}\cite[Chapter I, Proposition~ 10.8]{Hartshorne1970}\label{examplelemma} For any $m>0$, there is a one-to-one
correspondence between \begin{enumerate}
\item effective curves $Y$ on $X$, having no fibers as components,
of degree $m$ over $C$; and

\item sub-line bundles $L$ of $\mathrm{Sym}^{\otimes m} E$. \end{enumerate} Moreover, under this
correspondence, one has \begin{equation} D\cdot Y=
m\deg(E)-\deg(L)\label{examplekey}.\end{equation}\end{lemma}

\noindent For any effective curve $Y$ on
$X$, we denote by $m(Y)$ the degree of $Y$ over $C$. Then there
is an exact sequence
$$  0\rightarrow Pic(C)\stackrel{\pi}{\rightarrow}Pic(X)\stackrel{m}{\rightarrow}\bbZ\rightarrow0.$$
It follows that the the divisors on $X$, modulo numerical
equivalence, form a free abelian group of rank $2$, generated by $D$
and $F$ where $F$ is any fiber of $\bbP(E)$.

\begin{lemma}\cite[Chapter I, Theorem~ 10.5]{Hartshorne1970}  Let $C$ be a smooth curve of genus $g\geq  2$. \begin{enumerate}
\item    If $E$ is a stable vector bundle,  then every symmetric power $\mathrm{Sym}^{\otimes k}E$ is   semi-stable.

    \item  For any $r>0$ and $d\in \bbZ$, there exists a stable vector bundle with rank $r$ and degree $d$ such that all symmetric powers $\mathrm{Sym}^{\otimes k} E$ are stable.
\end{enumerate}
\end{lemma}

\begin{thm}\cite[Chapter I, Example~ 10.6]{Hartshorne1970} Let $E$ be a rank $2$ vector bundle over a smooth curve of genus $g\geq 2$. If $\deg(E)=0$ and  all symmetric powers $ \mathrm{Sym}^{\otimes k} E$ are stable, then $E$ is a strictly nef vector bundle, i.e. $\sO_E(1)$ is a strictly nef line bundle.  Moreover $E$ is not ample.
    \end{thm}
\begin{proof}
Let $Y$ be an arbitrary irreducible curve on $X$. If $Y$
is a fiber, then $D\cdot Y=1$. If $Y$ is an irreducible curve of
degree $m>0$ over $C$, then by Lemma \ref{examplelemma}, $Y$ is
corresponding to a subline bundle $L$ of $\mathrm{Sym}^{\otimes m}E$. Note that since
$\mathrm{Sym}^{\otimes m} E$ are stable and of degree zero for all $m\geq 1$, we have
$$\deg(L)<\frac{\deg\left(\mathrm{Sym}^{\otimes m}E\right)}{\mathrm{rank}\left(\mathrm{Sym}^{\otimes m}E\right)}=0.$$
 Therefore, by
formula (\ref{examplekey}) $$ D\cdot
Y=m\deg(E)-\deg(L)=-\deg(L)>0.$$ Hence, the line bundle $\sO_E(1)$
of the divisor $D$ is strictly nef, i.e. $E$ is a  strictly nef vector bundle. Since $\deg (E)=0$, $E$ can not be ample.
\end{proof}

\begin{cor} Let $E$ be a  vector bundle over a smooth curve $C$. If $E$ is stable and $\deg(E)=0$, then $E$ admits a Hermitian-flat metric.
    \end{cor}
\begin{proof}
Since $E$ is stable over $C$, there exists a
Hermitian-Einstein metric $h$ on $E$ (e.g. \cite{UhlenbeckYau1986}), i.e. $
g^{-1}\cdot R_{\alpha\bar \beta}= c\cdot h_{\alpha\bar \beta}$
for some constant $c$ where $g$ is a smooth metric on  $C$. Since
 $\deg(E)=0$, we deduce $c=0$, i.e. $(E,)$ is Hermitian-flat.
\end{proof}

\noindent We summarize Mumford's example as following.
\begin{example}\label{Mumfordexample} Let $C$ be a smooth curve of genus $g\geq 2$. There exists a rank $2$ vector bundle $E\rightarrow C$ satisfying the following properties:
    \begin{enumerate}
        \item $\deg(E)=0$;
        \item $\mathrm{Sym}^{\otimes k} E$ are stable for all $k\geq 1$;
        \item $E$ is strictly nef but not ample; $E^*$ is strictly nef but not ample;
        \item $E$ admits a Hermitian-flat metric.
\item
Let $X=\bbP(E)$, $\pi:X\rightarrow C$ be the projection and $\sO_{\bbP(E)}(1)$ be the tautological line bundle. Then $T_{X/C}=K^{-1}_{X/C}\cong \sO_{\bbP(E)}(2)\otimes \pi^*\det E^*$ is strictly nef.
    \end{enumerate}
\end{example}

 Although the strict nefness is not closed under tensor products and wedge products,  we still have the following   properties by using the Barton-Kleiman type criterion (Proposition \ref{cri0}).

\begin{prop}\label{basicproperties} Let $E$ and $F$ be two
vector bundles on a projective manifold $X$.
\begin{enumerate}
\item $E$ is a  strictly nef vector bundle if and only if for every smooth  curve
$C$ and for any non-constant morphism $f: C\rightarrow X$,    $f^*E$ is strictly nef.

\item If $E$ is  strictly nef, then any non-zero quotient bundle $Q$ of $E$ is  strictly nef.

\item $E\oplus F$  is  strictly nef if and only if both $E$ and $F$ are
strictly nef.

\item If the symmetric power $\mathrm{Sym}^{\otimes k} E$ is  strictly nef for some $k\geq 1$, then $E$ is  strictly nef.

\item Let  $f:Y\rightarrow X$ be  a  finite morphism such that  $Y$ is a  smooth projective variety. If $E$ is strictly nef, then so  is $f^*E$.

\item Let  $f:Y\rightarrow X$ be  a  surjective morphism such that  $Y$ is a  smooth projective variety. If $f^*E$ is strictly nef, then  $E$ is  strictly nef.

\end{enumerate}\end{prop}

\begin{example} Let $E$ be the strictly nef vector bundle  in Example \ref{Mumfordexample}.
    \begin{enumerate}
        \item  $\Lambda^2E=\det E$ is numerically trivial and it is not strictly nef;
        \item the tensor product $E\otimes E^*$ is not strictly nef since $H^0(C,E\otimes E^*)\cong \bbC$;

        \item  $ E\oplus E^*$ is strictly nef, but  $\mathrm{Sym}^{\otimes 2}(E\oplus E^*)$ is not strictly nef. 
    \end{enumerate}

\end{example}

\noindent The following result is well-known and it will be used frequently in the sequel.
\begin{lemma}\label{strictlynefandsemiample} If $L$ is a strictly nef and semi-ample line bundle, then it is ample.
\end{lemma}

\begin{proof} Since $L$ is semi-ample, $mL$ is globally generated for some large  $m$. Let  $\phi_m : X \rightarrow Y$ be
    the morphism defined by $|mL|$. Since $L$ is strictly nef,  it is easy to see that $\phi_m$ is finite  and so $L$ is ample.
\end{proof}

\begin{cor} Let $L$ be a strictly nef line bundle. Then $H^0(X,L^*)=0$. Moreover, for any line bundle $F$, there exists a postive integer $m_0=m_0(L,F)$ such that for $m\geqslant  m_0$
    \begin{equation}H^0(X,(L^*)^{\otimes m}\otimes F)=0.\label{formalvanishing}\end{equation}
\end{cor}
\begin{proof} Suppose $H^0(X,L^*)\neq 0$, then $L^*$ is effective. Since $L$ is nef, we deduce that $L$ is trivial and this is a contradiction. Hence $H^0(X,L^*)=0$ and $L^*$ is not pseudo-effective. By \cite{Yang2019}, the vanishing theorem (\ref{formalvanishing}) holds.
\end{proof}

\noindent Recall that for an ample line bundle $L$, one has the Kodaira vanishing theorem.
For an ample vector bundle E, one can only deduce $H^0(X,E^*)=0$ and the higher cohomology groups $H^i(X,E^*)$ ($i\geqslant  1$) may not vanish. For instance, when $X=\bbP^n$ with $n\geqslant  2$ and $E=T\bbP^n$, one has
$$H^1(X,E^*)\cong H^{1,1}(X,\bbC)\cong \bbC\neq 0.$$
For strictly nef vector bundles, we have a similar vanishing theorem.
\begin{thm} Let $E$ be a strictly nef vector bundle over a projective manifold $X$. Then $$H^0(X,E^*)=0.$$
\end{thm}
\begin{proof} Suppose $\sigma \in H^0(X,E^*)$ is a nonzero section. Then by \cite[Proposition~1.16]{DemaillyPeternellSchneider1994}, $\sigma$ does not vanish anywhere. This section gives a trivial subbundle of $E^*$ and so a trivial quotient bundle of the strictly nef vector bundle $E$. This contradicts to (2) of Proposition \ref{basicproperties}.
\end{proof}

\noindent Note that the  vanishing theorem in (\ref{formalvanishing}) does not hold for higher rank vector bundles. Indeed, let $E$ be the strictly nef vector bundle in Example \ref{Mumfordexample}, $X=\bbP(E)$ and $\pi:X\rightarrow C$ be the projection. Let $ F=K^{\otimes m}_C$ for some sufficiently large $m$. Then
$$H^0(C, \mathrm{Sym}^{\otimes k}E^*\otimes F)\neq 0$$
for all $k>0$.

\begin{remark} For a strictly nef vector bundle $E$ with rank $r\geqslant  2$, in general,  $H^0(X,\mathrm{Sym}^{\otimes k} E^*)=0$ dos not hold for $k\geqslant  2$.

\end{remark}

We give more examples on strictly nef vector bundle over higher
dimensional projective manifolds (see
\cite[Section~5]{LiuOuYang2020} for details). A line bundle $L$ over
a projective variety $X$ of dimension $n$ is called \emph{$k$-
strictly nef} if
$$L^{\dim Y}\cdot Y>0$$
for every irreducible subvariety $Y$ in $X$ with $0<\dim Y\leq k$.
Hence, $1$-strictly nef is exactly strictly nef, and an $n$-strictly
nef line bundle is ample.
\begin{thm}\cite[Lemma 3.2 and Theorem 6.1]{Subramanian1989}\label{Subramanian}
    Let $C$ be a smooth curve of genus $g\geqslant 2$. Then for any $r\geqslant 2$, there exists a Hermitian flat vector bundle $E$ of rank $r$ such that the tautological line bundle $\sO_{\bbP(E)}(1)$ is $(r-1)$-strictly nef. In particular, $E$ is strictly nef.
\end{thm}

\noindent  Fix  a smooth  curve $C$ of genus $g \geqslant 2$. Let  $r\geqslant 2$ and  $E$  be a vector bundle of rank $r$ provided in Theorem \ref{Subramanian}.

\begin{example}
    \label{example:subbundles}
    Let $X=\bbP(E)$. Then we have the following relative Euler exact sequence
    \[0\rightarrow \sO_X\rightarrow p^*E^*\otimes\sO_{\bbP(E)}(1)\rightarrow T_{X/C}\rightarrow 0,\]
    where $p\colon X=\bbP(E)\rightarrow C$ is the natural projection.
    It is shown in \cite[Example~5.9]{LiuOuYang2020} that $ p^*E^*\otimes \sO_{\bbP(E)}(1)$ is strictly nef.
\end{example}

\begin{example}
    We consider the following extension of vector bundles
    \[0\rightarrow Q\rightarrow G\rightarrow E^*\rightarrow 0,\]
    where $Q$ is a nef vector bundle of positive rank.    Since $E^*$ is Hermitian flat, it is numerically flat. In particular, $E^*$ is nef and so is $G$ (\cite[Proposition 1.15]{DemaillyPeternellSchneider1994}).
    Let $X=\bbP(G)$ and  $p: X= \bbP(G)\rightarrow C$ the natural projection. Then we have the following relative Euler sequence
    \[0\rightarrow \sO_X\rightarrow p^*G^*\otimes\sO_{\bbP(G)}(1)\rightarrow T_{X/C}\rightarrow 0.\]
    Since $E$ is a subbundle of $G^*$, it follows that
    $ p^*E\otimes \sO_{\bbP(G)}(1)$ is a subbundle of $p^*G^*\otimes \sO_{\bbP(G)}(1)$.
    We proved in \cite[Example~5.10]{LiuOuYang2020} that  $ p^*E\otimes \sO_{\bbP(G)}(1)$ is strictly nef and  the restriction of $F$ to fibers of $p$ is isomorphic to $\sO_{\bbP^d}(1)^{\oplus r}$. In particular,  $F$ is not a subbundle of $T_{X/C}$.
\end{example}

\vskip 2\baselineskip

\section{Strictly nef vector bundles}

 In this section, we consider strictly nef vector bundles over  higher dimensional projective manifolds.

\begin{thm} Let $E$ be a strictly nef vector bundle over a projective manifold $X$.   If $-K_X$ is nef and big, then $\det E$ is ample.
\end{thm}

\begin{proof} If $E$ is a strictly nef line bundle, then $E-K_X$ is  nef and big. By
    Kawamata-Reid-Shokurov base point free theorem, $E$ is semi-ample. Thanks to Lemma \ref{strictlynefandsemiample}, $E$ is ample.
    If $E$ has rank
     $r\geqslant 2$,   we consider the projective bundle $Y=\bbP(E)$.  Let  $\sO_E(1)$ be the tautological line bundle of the projection $\pi:Y\rightarrow X$. By the projection
    formula, we have $$ -K_Y=\sO_E(r)\otimes \pi^*(-K_X)\otimes \pi^*(\det
    E^*).$$
     For any $m>0$, the line bundle $L=\sO_E(m)\otimes \pi^*(\det
    E)$ is strictly nef. Since $$ L-K_Y=\sO_E(m+r)\otimes \pi^*(-K_X),$$
    we deduce $L-K_Y$ is strictly nef. On the other hand, $L-K_Y$ is big. Indeed,
    since both $\sO_{E}(1)$ and $-K_X$ are nef, the top intersection number \begin{eqnarray*}  (L-K_Y)^{n+
        r-1}&=&(\sO_E(m+r)\otimes \pi^*(-K_X))^{n+r-1}\\
    &\geq &\left(\sO_{E}(m+r)\right)^{r-1} \cdot
    \left(\pi^*(-K_X)\right)^{n}>0. \end{eqnarray*}
    Therefore, by the base
    point free theorem again, $L$ is semi-ample and so $L$ is ample. By the positivity of direct image sheaves (\cite{Mourougane1997}), we deduce that $\pi_*(K_{Y/X}\otimes L^{\otimes k})$ is ample for $k$ large enough.
    By using the projection formula, one can see that  the ample vector bundle $\pi_*(K_{Y/X}\otimes L^{\otimes k})$ is of the form  $\mathrm{Sym}^{\otimes k_0}E\otimes (\det E)^{\otimes k_1}$  where  $k_0$ and $k_1$ are some positive integers. In particular, $\det E$ is
    ample.
    \end{proof}

\noindent More generally, we propose the following conjecture.

\begin{conj}\label{LOYconjecture} Let $E$ be a strictly nef vector bundle over a projective manifold $X$.   If $-K_X$ is nef, then $\det E$ is ample.
\end{conj}

\noindent  It is known that every strictly nef line bundle over an abelian variety is ample (\cite[Proposition~1.4]{Serrano1995}), and Chaudhuri proved in \cite{Chaudhuri2020} that every strictly nef homogeneous bundle
on a complex flag variety is ample. We observe that Conjecture \ref{LOYconjecture} can be implied by Serrano's Conjecture \ref{Serranoconjecture1}.

\begin{prop}\label{SerranoconjectureimpliesLOYconjecture} Conjecture \ref{Serranoconjecture1} implies Conjecture \ref{LOYconjecture}.
\end{prop}
\begin{proof} Suppose  Conjecture \ref{Serranoconjecture1} is valid. Let $Y=\bbP(E)$, $L=\sO_{\bbP(E)}(1)$ and $\pi:Y\rightarrow X$ be the projection. For large $m$, $K_Y\otimes L^{\otimes m}$ is ample. Since $-K_X$ is nef,  $K_{Y/X}\otimes L^{\otimes m}$ is ample. We know $\det \pi_*\left(K_{Y/X}\otimes\left(K_{Y/X}\otimes L^{\otimes m}\right)\right)$ is  ample and so is $\det E$.
\end{proof}

 The following result is proved in \cite[Section~3]{LiOuYang2019}.
\begin{prop}\label{flat} Let $E$  be a strictly nef vector bundle
    over a projective manifold $X$. If either of the following holds
    \begin{enumerate}
        \item the Kodaira dimension $\kappa (X)$ satisfies $0 \leqslant \kappa (X)<\dim X$,

        \item $-K_X$ is pseudo-effective,
    \end{enumerate}  then  $\det E$ is not numerically trivial.
\end{prop}

\vskip 2\baselineskip

\section{Geometry of projective manifold with strictly nef bundles}

In this section, we describe the geometry related to strictly nef  and ample bundles.  As we mentioned before,     the following result is obtained  in \cite[Theorem~1.3]{LiOuYang2019}:
\begin{thm} Let $X$ be a projective manifold. If $T_X$ is strictly nef, then $X$ is isomorphic to a projective space.
    \end{thm}

\noindent Moreover,
a characterization of quadrics is  established in \cite[Theorem~1.5]{LiOuYang2019}, which is analogous to classical results   of  Cho-Miyaoka-Shepherd-Barron (\cite{ChoMiyaokaShepherd-Barron2002}) and Dedieu-Hoering (\cite{DedieuHoering2017}).
\begin{thm}
    Let $X$ be a  projective manifold of
    dimension $n\geq 3$. Suppose that $\bigwedge^2T_X$ is strictly nef,
    then $X$ is isomorphic to $\mathbb P^n$ or a quadric $\mathbb Q^n$.
\end{thm}

\noindent Gachet studied in \cite{Gachet2019}  the case when $\Lambda^3 T_X$ is strictly nef.  We have established that the tangent bundle    $T_X$ is
strictly nef if and only if  it is ample.  However, it is not valid for cotangent bundles.
\begin{example}\label{cotangentbundle} Let $X$ be a bidisk quotient, $\Delta\times\Delta/\Gamma$, with $\Gamma$ an irreducible torsion-free cocompact lattice. Let $E=T_X^*$ and $L$ be its tautological line bundle.   It is proved in \cite{ShepherdBarron1995} that $L$ is strictly nef and big, but it is not semi-ample.
\end{example}

\noindent  We propose the following problem which is analogous to the class result of Kobayashi that projective manifolds with ample cotangent bundle are hyperbolic.

\begin{problem} Let $X$ be a projective manifold. If $T_X^*$ is strictly nef,  is $X$ hyperbolic?
\end{problem}

  Let's consider manifolds with strictly nef canonical or anti-canonical bundles by recalling the Conjecture \ref{CampanaPeternellConjecture1} of Campana and Peternell.

\begin{conj} \label{CampanaPeternellConjecture} Let $X$ be a    projective
    manifold. If $K_X^{-1}$ is strictly nef, then $X$
    is  Fano.
\end{conj}

\noindent  Recently,  some evidences  are  established in \cite[Theorem~1.2]{LiOuYang2019}.

\begin{thm}\label{LOYrationallyconnected} Let $X$ be a projective manifold.  If $\Lambda^rT_X$ is strictly nef, then $X$ is rationally connected. In particular, if $K^{-1}_X$ is strictly nef, then $X$ is rationally connected.
    \end{thm}

\noindent Let $f:X\rightarrow Y$ be a smooth morphism between two projective manifolds. It is well-known that if $K^{-1}_X$ is ample, then so is $K^{-1}_Y$ (\cite{KollarMiyaokaMori1992}, see also \cite{BirkarChen2016} for semi-ampleness).  It is natural to propose the following conjecture.
\begin{conj}\label{directimage} Let $f:X\rightarrow Y$ be a smooth morphism between two projective manifolds. If $K^{-1}_X$ is strictly nef, then so is $K^{-1}_Y$.
\end{conj}

\noindent  Indeed, this conjecture can be regarded as a consequence of Conjecture \ref{CampanaPeternellConjecture}.  Thanks  to Theorem \ref{LOYrationallyconnected}, one obtains   a partial answer to Conjecture \ref{directimage}.

\begin{cor}Let $f:X\rightarrow Y$ be a smooth morphism between two projective manifolds. If $K^{-1}_X$ is strictly nef, then $Y$ is rationally connected.
\end{cor}

\begin{example} Let $f:X\rightarrow Y$ be a smooth morphism between two projective manifolds. It is well-known that $K^{-1}_{X/Y}$ can not be ample (\cite{KollarMiyaokaMori1992}). However, it can be strictly nef.
\end{example}

\noindent We also propose the following general conjecture concerning strictly nef bundles.

\begin{conj}Let $X$ be a projective manifold. \begin{enumerate} \item If $\Lambda^r T_X$ is strictly nef for some $r>0$, then $K^{-1}_X$ is ample;
        \item       If $\Lambda^r T^*_X$ is strictly nef  for some $r>0$, then $K_X$ is ample.

    \end{enumerate}
\end{conj}
\noindent When $T^*_X$ is strictly nef, it is of particular interest  and it is also related to the Kobayashi-Lang conjecture on hyperbolicity.\\

Let's consider the geometry of projective manifolds whose tangent bundle contains a "positive" subsheaf. Recall that,  Andreatta and Wi\'sniewski obtained in \cite[Theorem]{AndreattaWisniewski2001}
the following characterization of projective spaces.
\begin{thm}
    \label{thm:Andreatta-Wisniewski}
    Let $X$ be an $n$-dimensional   projective manifold.
    Assume that the tangent bundle $T_X$ contains a locally free ample subsheaf $\sF$ of rank $r$. Then $X\cong  \bbP^n$ and either $\sF\cong  T_{\bbP^n}$  or $\sF\cong \sO_{\bbP^n}(1)^{\oplus r}$.
\end{thm}

\noindent  According to Example \ref{Mumfordexample},    Theorem \ref{thm:Andreatta-Wisniewski}
does not hold  if the subsheaf $\sF$ is  assumed to be strictly nef. Indeed, we  obtained in \cite[Theorem~1.3]{LiuOuYang2020} the following structure theorem for projective
manifolds whose tangent bundle contains a strictly nef subsheaf.

\begin{thm}
    \label{thm:main-theorem}
    Let $X$ be a  projective manifold. Assume that the tangent bundle $T_X$ contains a locally free strictly nef subsheaf $\sF$ of rank $r$.
    Then $X$ admits a $\bbP^d$-bundle structure $\varphi: X\rightarrow T$ for some $d\geq r$.
    Moreover,
    $T$ is a hyperbolic projective manifold.
\end{thm}

 \noindent  Actually,  we  obtained in \cite[Theorem~8.1]{LiuOuYang2020}  a  concrete description on the
 structure of the subsheaf $\sF$ and it is exactly one of the following:
 \begin{enumerate}
    \item $\sF\cong  T_{X/T}$ and $X$ is isomorphic to a flat projective bundle over
    $T$;
    \item $\sF$ is a numerically projectively flat vector bundle and its restriction  on every fiber of $\varphi$ is  isomorphic to  $\sO_{\bbP^d}(1)^{\oplus r}$.
 \end{enumerate}

 \noindent   When $\dim T=0$, we obtained in \cite[Theorem~1.4]{LiuOuYang2020} a new characterization of projective spaces.
\begin{cor}
    \label{simply-connected-Pn}  Let $X$ be an $n$-dimensional  projective manifold such that $T_X$ contains a locally free strictly nef subsheaf $\sF$.
    If $\pi_1(X)$ is virtually abelian, then $X$ is isomorphic to  $\bbP^n$, and $\sF$ is isomorphic to either $T_{\bbP^n}$ or $\sO_{\bbP^n}(1)^{\oplus r}$.
\end{cor}

\noindent  When $\dim T>0$, we established in \cite[Corollary~1.5]{LiuOuYang2020} the  existence of non-zero symmetric differentials.

\begin{cor}
    \label{cor:existence-symmetric-forms}
    Let $X$ be a projective manifold whose tangent bundle contains a  locally free strictly nef subsheaf. If $X$ is not isomorphic to a projective space, then $X$ has a non-zero symmetric differential, i.e. $H^0(X,\Sym^i\Omega_X)\not=0$ for some $i>0$.
\end{cor}

\def\cprime{$'$} 
\renewcommand\refname{Reference}
\bibliographystyle{alpha}
\bibliography{strcitlynef}

\newcommand{\etalchar}[1]{$^{#1}$}
\begin{thebibliography}{MOSC{\etalchar{+}}15}

\bibitem[ADK08]{AraujoDruelKovacs2008}
Carolina Araujo, St{\'e}phane Druel, and S{\'a}ndor~J. Kov{\'a}cs.
\newblock Cohomological characterizations of projective spaces and
  hyperquadrics.
\newblock {\em Invent. Math.}, 174(2):233--253, 2008.

\bibitem[AKP08]{AproduKebekusPeternell2008}
Marian Aprodu, Stefan Kebekus, and Thomas Peternell.
\newblock Galois coverings and endomorphisms of projective varieties.
\newblock {\em Math. Z.}, 260(2):431--449, 2008.

\bibitem[Ara06]{Araujo2006}
Carolina Araujo.
\newblock Rational curves of minimal degree and characterizations of projective
  spaces.
\newblock {\em Math. Ann.}, 335(4):937--951, 2006.

\bibitem[AW01]{AndreattaWisniewski2001}
Marco Andreatta and Jaros{\l}aw~A. Wi{\'s}niewski.
\newblock On manifolds whose tangent bundle contains an ample subbundle.
\newblock {\em Invent. Math.}, 146(1):209--217, 2001.

\bibitem[BC16]{BirkarChen2016}
Caucher Birkar and Yifei Chen.
\newblock Images of manifolds with semi-ample anti-canonical divisor.
\newblock {\em J. Algebraic Geom.}, 25(2):273--287, 2016.

\bibitem[BHN19]{Biswas2019}
Indranil Biswas, Krishna Hanumanthu, and Donihakkalu~S. Nagaraj.
\newblock Positivity of vector bundles on homogeneous varieties.
\newblock {\em ArXiv preprint}, 1904.09310, 2019.

\bibitem[Cao16]{Cao2016}
Junyan Cao.
\newblock Albanese maps of projective manifolds with nef anticanonical bundles.
\newblock {\em ArXiv preprint}, 1612.05921, 2016.

\bibitem[CCP08]{CampanaChenPeternell2008}
Fr\'{e}d\'{e}ric Campana, Jungkai~A. Chen, and Thomas Peternell.
\newblock Strictly nef divisors.
\newblock {\em Math. Ann.}, 342(3):565--585, 2008.

\bibitem[CCP19]{CampanaCaoPaun2019}
Fr{\'e}d{\'e}ric Campana, Junyan Cao, and Mihai P{\u{a}}un.
\newblock Subharmonicity of direct images and applications.
\newblock {\em arXiv preprint arXiv:1906.11317}, 2019.

\bibitem[CH19]{CaoHoering2019}
Junyan Cao and Andreas H{\"o}ring.
\newblock A decomposition theorem for projective manifolds with nef
  anticanonical bundle.
\newblock {\em Journal of Algebraic Geometry}, 28(3):567--597, 2019.

\bibitem[Cha20]{Chaudhuri2020}
Priyankur Chaudhuri.
\newblock Strictly nef divisors and some remarks on a conjecture of serrano.
\newblock {\em ArXiv preprint}, 2008.05009, 2020.

\bibitem[CMSB02]{ChoMiyaokaShepherd-Barron2002}
Koji Cho, Yoichi Miyaoka, and Nicholas~I. Shepherd-Barron.
\newblock Characterizations of projective space and applications to complex
  symplectic manifolds.
\newblock In {\em Higher dimensional birational geometry ({K}yoto, 1997)},
  volume~35 of {\em Adv. Stud. Pure Math.}, pages 1--88. Math. Soc. Japan,
  Tokyo, 2002.

\bibitem[CP91]{CampanaPeternell1991}
Fr{\'e}d{\'e}ric Campana and Thomas Peternell.
\newblock Projective manifolds whose tangent bundles are numerically effective.
\newblock {\em Math. Ann.}, 289(1):169--187, 1991.

\bibitem[CP98]{CampanaPeternell1998}
Fr{\'e}d{\'e}ric Campana and Thomas Peternell.
\newblock Rational curves and ampleness properties of the tangent bundle of
  algebraic varieties.
\newblock {\em Manuscripta Math.}, 97(1):59--74, 1998.

\bibitem[DH17]{DedieuHoering2017}
Thomas Dedieu and Andreas H{\"o}ring.
\newblock Numerical characterisation of quadrics.
\newblock {\em Algebr. Geom.}, 4(1):120--135, 2017.

\bibitem[DPS94]{DemaillyPeternellSchneider1994}
Jean-Pierre Demailly, Thomas Peternell, and Michael Schneider.
\newblock Compact complex manifolds with numerically effective tangent bundles.
\newblock {\em J. Algebraic Geom.}, 3(2):295--345, 1994.

\bibitem[DPS01]{DemaillyPeternellSchneider2001}
Jean-Pierre Demailly, Thomas Peternell, and Michael Schneider.
\newblock Pseudo-effective line bundles on compact {K}\"ahler manifolds.
\newblock {\em Internat. J. Math.}, 12(6):689--741, 2001.

\bibitem[Dru04]{Druel2004}
St{\'e}phane Druel.
\newblock Caract\'erisation de l'espace projectif.
\newblock {\em Manuscripta Math.}, 115(1):19--30, 2004.

\bibitem[FLW17]{FengLiuWan2017}
Huitao Feng, Kefeng Liu, and Xueyuan Wan.
\newblock Compact {K}\"{a}hler manifolds with positive orthogonal bisectional
  curvature.
\newblock {\em Math. Res. Lett.}, 24(3):767--780, 2017.

\bibitem[Gac19]{Gachet2019}
Cécile Gachet.
\newblock Positivité du fibré tangent et de ses puissances extérieures.
\newblock Master's thesis, Paris 6, 2019.

\bibitem[Gra62]{Grauert1962}
Hans Grauert.
\newblock \"{U}ber {M}odifikationen und exzeptionelle analytische {M}engen.
\newblock {\em Math. Ann.}, 146:331--368, 1962.

\bibitem[Gri65]{Griffiths1965}
Phillip~A. Griffiths.
\newblock Hermitian differential geometry and the theory of positive and ample
  holomorphic vector bundles.
\newblock {\em J. Math. Mech.}, 14:117--140, 1965.

\bibitem[Gri69]{Griffiths1969}
Phillip~A. Griffiths.
\newblock Hermitian differential geometry, {C}hern classes, and positive vector
  bundles.
\newblock In {\em Global {A}nalysis ({P}apers in {H}onor of {K}. {K}odaira)},
  pages 185--251. Univ. Tokyo Press, Tokyo, 1969.

\bibitem[Har66]{Hartshorne1966}
Robin Hartshorne.
\newblock Ample vector bundles.
\newblock {\em Inst. Hautes \'{E}tudes Sci. Publ. Math.}, (29):63--94, 1966.

\bibitem[Har70]{Hartshorne1970}
Robin Hartshorne.
\newblock {\em Ample subvarieties of algebraic varieties}.
\newblock Lecture Notes in Mathematics, Vol. 156. Springer-Verlag, Berlin-New
  York, 1970.
\newblock Notes written in collaboration with C. Musili.

\bibitem[HMP10]{Hering2010}
Milena Hering, Mircea Musta\c{t}\u{a}, and Sam Payne.
\newblock Positivity properties of toric vector bundles.
\newblock {\em Ann. Inst. Fourier (Grenoble)}, 60(2):607--640, 2010.

\bibitem[Hwa06]{Hwang2006}
Jun-Muk Hwang.
\newblock Rigidity of rational homogeneous spaces.
\newblock In {\em International {C}ongress of {M}athematicians. {V}ol. {II}},
  pages 613--626. Eur. Math. Soc., Z\"urich, 2006.

\bibitem[KMM92]{KollarMiyaokaMori1992}
J{\'a}nos Koll{\'a}r, Yoichi Miyaoka, and Shigefumi Mori.
\newblock Rational connectedness and boundedness of {F}ano manifolds.
\newblock {\em J. Differential Geom.}, 36(3):765--779, 1992.

\bibitem[KO73]{KobayashiOchiai1973}
Shoshichi Kobayashi and Takushiro Ochiai.
\newblock Characterizations of complex projective spaces and hyperquadrics.
\newblock {\em J. Math. Kyoto Univ.}, 13:31--47, 1973.

\bibitem[Laz04a]{Lazarsfeld2004}
Robert Lazarsfeld.
\newblock {\em Positivity in algebraic geometry. {I}. Classical setting: line
  bundles and linear series}, volume~48 of {\em Ergebnisse der Mathematik und
  ihrer Grenzgebiete. 3. Folge. A Series of Modern Surveys in Mathematics
  [Results in Mathematics and Related Areas. 3rd Series. A Series of Modern
  Surveys in Mathematics]}.
\newblock Springer-Verlag, Berlin, 2004.

\bibitem[Laz04b]{Lazarsfeld2004a}
Robert Lazarsfeld.
\newblock {\em Positivity in algebraic geometry II: Positivity for vector
  bundles, and multiplier ideals.}, volume~49 of {\em Ergebnisse der Mathematik
  und ihrer Grenzgebiete. 3. Folge. A Series of Modern Surveys in Mathematics
  [Results in Mathematics and Related Areas. 3rd Series. A Series of Modern
  Surveys in Mathematics]}.
\newblock Springer-Verlag, Berlin, 2004.

\bibitem[Li18]{Li2018}
Ping Li.
\newblock The spectral rigidity of complex projective spaces, revisited.
\newblock {\em Math. Z.}, 290(3-4):1115--1143, 2018.

\bibitem[Liu19]{Liu2019}
Jie Liu.
\newblock Characterization of projective spaces and $\mathbb{P}^{r}$ -bundles
  as ample divisors.
\newblock {\em Nagoya Math. J.}, 233:155--169, 2019.

\bibitem[LOY19]{LiOuYang2019}
Duo Li, Wenhao Ou, and Xiaokui Yang.
\newblock On projective varieties with strictly nef tangent bundles.
\newblock {\em J. Math. Pures Appl. (9)}, 128:140--151, 2019.

\bibitem[LOY20]{LiuOuYang2020}
Jie Liu, Wenhao Ou, and Xiaokui Yang.
\newblock Projective manifolds whose tangent bundle contains a strictly nef
  subsheaf.
\newblock {\em ArXiv preprint}, 2004.08507, 2020.

\bibitem[LOY21]{LiuOuYang2}
Jie Liu, Wenhao Ou, and Xiaokui Yang.
\newblock Projective manifolds with a strictly nef pair.
\newblock {\em preprint}, 2021.

\bibitem[LP20a]{LazicPeternell2020}
Vladimir Lazi\'{c} and Thomas Peternell.
\newblock On {G}eneralised {A}bundance, {I}.
\newblock {\em Publ. Res. Inst. Math. Sci.}, 56(2):353--389, 2020.

\bibitem[LP20b]{LazicPeternell2020a}
Vladimir Lazi\'{c} and Thomas Peternell.
\newblock On {G}eneralised {A}bundance, {II}.
\newblock {\em Peking Math. J.}, 3(1):1--46, 2020.

\bibitem[LR94]{LanteriRondena1994}
Antonio Lanteri and Barbara Rondena.
\newblock Numerically positive divisors on algebraic surfaces.
\newblock {\em Geom. Dedicata}, 53(2):145--154, 1994.

\bibitem[Mae93]{Maeda1993}
Hidetoshi Maeda.
\newblock A criterion for a smooth surface to be {D}el {P}ezzo.
\newblock {\em Math. Proc. Cambridge Philos. Soc.}, 113(1):1--3, 1993.

\bibitem[Mok88]{Mok1988}
Ngaiming Mok.
\newblock The uniformization theorem for compact {K}\"ahler manifolds of
  nonnegative holomorphic bisectional curvature.
\newblock {\em J. Differential Geom.}, 27(2):179--214, 1988.

\bibitem[Mor79]{Mori1979}
Shigefumi Mori.
\newblock Projective manifolds with ample tangent bundles.
\newblock {\em Ann. of Math. (2)}, 110(3):593--606, 1979.

\bibitem[MOSC{\etalchar{+}}15]{MunozOcchettaSolaCondeWatanabeEtAl2015}
Roberto Mu{\~n}oz, Gianluca Occhetta, Luis~Eduardo Sol{\'a}~Conde, Kiwamu
  Watanabe, and Jaros{\l}aw~A. Wi{\'s}niewski.
\newblock A survey on the {C}ampana-{P}eternell conjecture.
\newblock {\em Rend. Istit. Mat. Univ. Trieste}, 47:127--185, 2015.

\bibitem[Mou97]{Mourougane1997}
Christophe Mourougane.
\newblock Images directes de fibr\'es en droites adjoints.
\newblock {\em Publ. Res. Inst. Math. Sci.}, 33(6):893--916, 1997.

\bibitem[Pet90]{Peternell1990}
Thomas Peternell.
\newblock A characterization of {${\bf P}_n$} by vector bundles.
\newblock {\em Math. Z.}, 205(3):487--490, 1990.

\bibitem[Pet96]{Peternell1996}
Thomas Peternell.
\newblock Manifolds of semi-positive curvature.
\newblock In {\em Transcendental methods in algebraic geometry ({C}etraro,
  1994)}, volume 1646 of {\em Lecture Notes in Math.}, pages 98--142. Springer,
  Berlin, 1996.

\bibitem[SB95]{ShepherdBarron1995}
N.~I. Shepherd-Barron.
\newblock Infinite generation for rings of symmetric tensors.
\newblock {\em Math. Res. Lett.}, 2(2):125--128, 1995.

\bibitem[Ser95]{Serrano1995}
Fernando Serrano.
\newblock Strictly nef divisors and {F}ano threefolds.
\newblock {\em J. Reine Angew. Math.}, 464:187--206, 1995.

\bibitem[Sub89]{Subramanian1989}
Swaminathan Subramanian.
\newblock Mumford's example and a general construction.
\newblock {\em Proc. Indian Acad. Sci. Math. Sci.}, 99(3):197--208, 1989.

\bibitem[SY80]{SiuYau1980}
Yum~Tong Siu and Shing~Tung Yau.
\newblock Compact {K}\"ahler manifolds of positive bisectional curvature.
\newblock {\em Invent. Math.}, 59(2):189--204, 1980.

\bibitem[Tos17]{Tosatti2017}
Valentino Tosatti.
\newblock Uniqueness of {$\Bbb{C}\Bbb{P}^n$}.
\newblock {\em Expo. Math.}, 35(1):1--12, 2017.

\bibitem[Ueh00]{Uehara2000}
Hokuto Uehara.
\newblock On the canonical threefolds with strictly nef anticanonical divisors.
\newblock {\em J. Reine Angew. Math.}, 522:81--91, 2000.

\bibitem[UY86]{UhlenbeckYau1986}
Karen~Keskulla Uhlenbeck and Shing-Tung Yau.
\newblock On the existence of {H}ermitian-{Y}ang-{M}ills connections in stable
  vector bundles.
\newblock volume~39, pages S257--S293. 1986.
\newblock Frontiers of the mathematical sciences: 1985 (New York, 1985).

\bibitem[Wah83]{Wahl1983}
Jonathan~M. Wahl.
\newblock A cohomological characterization of {${\bf P}^{n}$}.
\newblock {\em Invent. Math.}, 72(2):315--322, 1983.

\bibitem[Yan17]{Yang2017}
Xiaokui Yang.
\newblock Big vector bundles and complex manifolds with semi-positive tangent
  bundles.
\newblock {\em Math. Ann.}, 367(1-2):251--282, 2017.

\bibitem[Yan19]{Yang2019}
Xiaokui Yang.
\newblock A partial converse to the {A}ndreotti-{G}rauert theorem.
\newblock {\em Compos. Math.}, 155(1):89--99, 2019.

\bibitem[Yau74]{Yau1974}
Shing~Tung Yau.
\newblock On the curvature of compact {H}ermitian manifolds.
\newblock {\em Invent. Math.}, 25:213--239, 1974.

\end{thebibliography}

\end{document}